\documentclass[12pt]{amsart}
\usepackage{amsmath,amsthm,amsfonts,latexsym,amssymb,amscd,color}
\usepackage{chemarr}
\newtheorem{thm}{Theorem}[section] 

\newtheorem{lm}[thm]{Lemma}

\newtheorem{clm}[thm]{Claim}
\newtheorem*{clm*}{Claim}
\theoremstyle{definition}
\newtheorem{df}[thm]{Definition}

\numberwithin{equation}{section}
\newcommand{\sprf}{\noindent{\it Proof.}} 
\newcommand{\sqed}{\hfill\rule{1.3mm}{3mm}\medskip}

\newcommand{\cproof}{\noindent{\it Proof of Claim.}\ } 
\newcommand{\cqed}{\hfill\rule{1.3mm}{3mm}}

\newcommand{\set}[2]{\{#1\mid\nobreak\text{#2}\}} 
\DeclareMathOperator{\id}{\textsl{id}}     

\newcommand{\wec}[1]{{\mathbf{#1}}}  
\newcommand{\m}[1]{{\mathbf{\uppercase{#1}}}}

\DeclareMathOperator{\Ho}{\mathsf{H}}
\DeclareMathOperator{\Su}{\mathsf{S}}

\DeclareMathOperator{\Pd}{\mathsf{P}}

\DeclareMathOperator{\Con}{Con}
\DeclareMathOperator{\Conb}{\mathbf{Con}} 
\DeclareMathOperator{\cg}{Cg}

\newcommand{\bd}{\begin{description}}
\newcommand{\ed}{\end{description}}
\begin{document}

\title{Varieties whose finitely generated members are free}

\author{Keith A. Kearnes}
\address[Keith Kearnes]{Department of Mathematics\\
University of Colorado\\
Boulder, CO 80309-0395\\
USA}
\email{keith.kearnes@colorado.edu}
\author{Emil W. Kiss}
\address[Emil W. Kiss]{
Lor\'{a}nd E{\"o}tv{\"o}s University\\
Department of Algebra and Number Theory\\
H--1117 Budapest, P\'{a}zm\'{a}ny P\'{e}ter s\'{e}t\'{a}ny 1/c.\\
Hungary}
\email{ewkiss@cs.elte.hu}
\author{\'Agnes Szendrei}
\address[\'Agnes Szendrei]{Department of Mathematics\\
University of Colorado\\
Boulder, CO 80309-0395\\
USA}
\email{agnes.szendrei@colorado.edu}
\thanks{%
This material is based upon work supported by
the National Science Foundation grant no.\ DMS 1500254 and
the Hungarian National Foundation for Scientific Research (OTKA)
grant no.\ K104251 and K115518.
}

\subjclass{08B20 (08A05, 03C35)}
\keywords{}

\begin{abstract}
  We prove that a
  variety of algebras whose finitely
  generated members are free
  must be definitionally equivalent to
  the variety of sets, the variety of pointed sets, a variety of vector spaces
  over a division ring, or a variety of affine vector spaces over a
  division ring.
\end{abstract}

\maketitle

\section{Introduction}\label{intro_sec}
In this paper we address a MathOverflow question, \cite{campion},
which asks for a description of the varieties
where every algebra is free, as well as a description of the varieties
satisfying the weaker requirement that
every finitely generated algebra is free.

Steven Givant classified the varieties where every algebra is free
in \cite{givant}.
He proved that they
are precisely those definitionally equivalent to
\begin{itemize}
\item the variety of sets,
\item the variety of pointed sets,
\item a variety of vector spaces over a division ring, or
\item a variety of affine spaces over a division ring.
\end{itemize}
In this paper we use different techniques
to classify the varieties where every
finitely generated algebra is free.
Our result is that if the finitely generated members of
a variety
$\mathcal V$ are free, then $\mathcal V$
must also be one of these types of varieties
(sets, pointed sets, vector spaces or affine spaces).
Hence, if the finitely generated algebras
in $\mathcal V$ are free, then all algebras in $\mathcal V$ are free.
This gives a new proof of Givant's Theorem under weaker
hypotheses.

In the last section of the paper we discuss
some variations on the main question.
First we consider
a ``large rank'' variation:
Which varieties have the property
that their finitely generated algebras
of sufficiently large rank
are free?
That is, 
for which varieties $\mathcal V$ is there a finite
number $k$ such that every finitely generated
algebra in $\mathcal V$ requiring more than $k$
generators is free?
We prove the theorem that a locally finite variety with this property
must even have the property that all of its
nonsingleton algebras are free,
and it is essentially
one of the four types of varieties discussed above.
Without the assumption of local finiteness this theorem fails.

Next we examine a ``small rank'' variation of the main question:
Is there some $n$ such that,
if all $(\leq n)$-generated algebras in a variety 
are free, then all finitely generated algebras
in the variety
are free?
The answer to this is negative.
We show that 
for each positive integer $n$ there 
exist varieties
in which the algebras generated by at most $n$ elements are free,
but the  $(n+1)$-generated algebras are not all free.

\section{Abelian and affine algebras}
Please refer to \cite{freese-mckenzie, hobby-mckenzie, kearnes-kiss}
for elaboration of the introductory remarks of this section.

An algebra $\m a$
is \emph{abelian} if it satisfies the \emph{term condition},
which is the assertion that if $t(\wec{x},\wec{y})$ is a term
in the language, $\wec{a}, \wec{b}, \wec{u}$ and $\wec{v}$
are tuples of elements of $A$, and
\[
t^{\m a}(\underline{\wec{a}},\wec{u}) = t^{\m a}(\underline{\wec{a}},\wec{v}), 
\]
then 
\[
t^{\m a}(\underline{\wec{b}},\wec{u}) = t^{\m a}(\underline{\wec{b}},\wec{v}). 
\]
This property
is the same as the property that the diagonal 
$\set{(a,a)}{$a\in A$}$
of $\m a\times \m a$
is the class of a congruence.

An algebra $\m b$ is \emph{affine} if it is polynomially equivalent
to a module. This means that there is a ring $R$ and a left
$R$-module structure 
${}_R B$
on the universe $B$ of $\m b$ such that
the polynomial operations of $\m b$ coincide with the
$R$-module polynomial operations 
of ${}_R B$.
(A polynomial
operation of an algebra $\m b$ is an operation
$p(\wec{x})$
obtained from a term operation by substituting constants for
some of the variables, i.e.\  
$p(\wec{x}) = t^{\m b}(\wec{x},\wec{b})$
for some term $t(\wec{x},\wec{y})$ in the language and some
tuple $\wec{b}$ of elements of $B$.)

A variety is abelian or affine if its members are.
It is a fact that affine algebras and varieties are abelian,
but the converse is false, e.g.\ unary varieties are abelian
but not affine.

Abelian varieties that are not affine are poorly understood
at present. If $\mathcal V$ is a \underline{locally finite} variety that 
is abelian but not affine, then it can be proved that $\mathcal V$
contains a very ``bad'' or ``structureless'' algebra, i.e.\ 
one that is definitionally equivalent
to a matrix power of a 
two-element set or pointed set.
The procedure for proving this is to first exploit
the nonaffineness assumption to construct
a finite ``strongly abelian'' algebra $\m s\in \mathcal V$,
and then to examine a minimal subvariety of the variety
$\Ho\Su\Pd(\m s)$
generated by $\m s$.
The structure of such minimal subvarieties are determined
by the classification theorem for minimal abelian varieties,
which can be found in 
\cite{kkv1} and \cite{szendrei}. Namely, a minimal subvariety
of a variety generated by a finite strongly abelian algebra
is definitionally
equivalent to a matrix power of the variety of sets
or the variety of pointed sets.

These arguments fail at the very first step
for varieties that are not locally finite: 
it is not known 
if the construction discussed in the preceding paragraph
yields an algebra $\m s$ that is strongly abelian. 
In this section we examine
the construction of $\m s$
and identify some ``strongly abelian--like''
properties of $\m s$.

First, a congruence $\theta\in\Con(\m a)$ is \emph{strongly abelian}
if it satisfies the strong term condition,
which is the assertion that if $t(\wec{x},\wec{y})$ is a term
in the language, 
$\wec{a}, \wec{b}, \wec{c}, \wec{u}, \wec{v}$
are tuples of elements of $A$ with 
$\wec{a}$, $\wec{b}$, $\wec{c}$\ 
$\theta$-related coordinatewise and
$\wec{u}, \wec{v}$\ 
$\theta$-related coordinatewise, and
\[
t^{\m a}(\underline{\wec{a}},\wec{u}) = t^{\m a}(\underline{\wec{b}},\wec{v}), 
\]
then
\[
t^{\m a}(\underline{\wec{c}},\wec{u}) = t^{\m a}(\underline{\wec{c}},\wec{v}). 
\]

Now suppose that $\m a$ is abelian
and $\theta\in\Con(\m a)$ is strongly abelian.
The construction we are concerned with is the following one:
Let $\m a(\theta)$ be the subalgebra of $\m a\times \m a$
supported by (the graph of) $\theta$,
that is, $\set{(a,b)\in A}{$a\equiv_\theta b $}$.
Let $\Delta$ be the congruence on $\m a(\theta)$ 
generated by $D\times D$ where $D = \{(a,a)\;|\;a\in A\}$
is the diagonal. 
$D$ is a $\Delta$-class, because $\m a$ is abelian.
Let $\m s = \m s_{\m a,\theta} := \m a(\theta)/\Delta$.
Let $0 = D/\Delta\in S$.

If $\m a$ is a finite member of an abelian variety,
then it is possible to
prove that the resulting algebra $\m s$ is
a strongly abelian
member of the variety (meaning that
all of its congruences
are strongly
abelian). Without finiteness we do not know how to prove this.
However, we can prove the following.

\begin{lm}\label{nonaffine}
  Let $\mathcal V$ be an abelian variety, and suppose that
  $\theta$ is a nontrivial strongly abelian congruence on some
  $\m a\in\mathcal V$. Let $\m s = \m s_{\m a,\theta}$
  and let $0 = D/\Delta\in S$. The following are true:
  \begin{enumerate}
    \item $\m s$ has more than one element.
  \item $\{0\}$ is a 1-element subuniverse of\/ $\m s$.
  \item $\m s$ has ``Property~P'': for every $n$-ary polynomial
    $p(\wec{x})$ of\/ $\m s$ and every tuple $\wec{s}\in S^n$
    \[
p(\wec{s})=0\quad\textrm{implies}\quad p(\wec{0})=0,
\]
where $\wec{0} = (0,0,\ldots,0)$.
  \item Whenever $t(x_1,\ldots,x_n)$ is
     a $\mathcal V$-term, and
    \[
    \mathcal V\models t(\wec{x})=
    t(\wec{y})
    \]
    where $\wec{x}$ and $\wec{y}$ are tuples of not necessarily
    distinct variables which differ in the $i$th position,
    then the term operation $t^{\m s}(x_1,\ldots,x_n)$ is
    independent of its $i$th variable.
  \item $\m s$ has a congruence $\sigma$ such that the
    algebra $\m s/\sigma$ satisfies (1)--(4) of this lemma,
    and $\m s/\sigma$ also
    has a compatible partial order $\leq$ such that
$0\leq s$ for every $s\in S/\sigma$.
    \end{enumerate}
  \end{lm}

\begin{proof}
  
  [Item (1)]
    Since $\m a$ is abelian, the diagonal 
  $D$ is the class of a congruence on $\m a\times \m a$,
  namely the congruence generated by $D\times D$. 
  This congruence restricts to $\m a(\theta)$ to have $D$ as a class.
  Since $\theta$ is nontrivial, it properly contains $D$,
  so the congruence $\Delta$ of
  $\m a(\theta)$ generated by $D\times D$ is proper. Equivalently, 
  $\m s = \m a(\theta)/\Delta$ is nontrivial.
  
\bigskip

  [Item (2)]
Since $D$ is a subuniverse of $\m a(\theta)$, $\{D/\Delta\}=\{0\}$
is a subuniverse of $\m s$.

\bigskip

[Item (3)] To show that $\m s$ has Property~P,
choose  $p(\wec{x})$ and $\wec{s}\in S^n$ 
such that $p(\wec{s})=0$. Our goal
is to show that $p(\wec{0}) = 0$.

Express $p(\wec{x})$ as $t^{\m s}(\wec{x},\wec{u})$
for some 
term $t(\wec{x},\wec{y})$ and for some
tuple $\wec{u}$ with coordinates in  $S$.
Also, express 
the coordinates $s_i$ and $u_j$ 
of the tuples $\wec{s}$ and $\wec{u}$
as $s_i = (a_i,b_i)/\Delta$ and $u_j = (v_j,w_j)/\Delta$ where
$(a_i,b_i), (v_j,w_j)\in\theta$. Then $p(\wec{s})=0$ may be expressed as
$t^{\m a(\theta)}\big((\wec{a},\wec{b}),(\wec{v},\wec{w})\big)\in D$,
or 
\[
t^{\m a}(\underline{\wec{a}},\wec{v})=t^{\m a}(\underline{\wec{b}},\wec{w}).
\]
Since $\theta$ is strongly abelian, by 
the strong term condition we derive that
\[
t^{\m a}(\underline{\wec{a}},\wec{v})=t^{\m a}(\underline{\wec{a}},\wec{w})
\]
holds,
which may be expressed as
$t^{\m a(\theta)}((\wec{a},\wec{a}),(\wec{v},\wec{w}))\in D$,
or $p(\wec{0})=p((\wec{a},\wec{a})/\Delta)=0$.

\bigskip

[Item (4)]
Assume for the sake of
 simplicity that $i=1$ in the statement of (4),
that is,
${\mathcal V}\models t(x,\wec{w})=t(y,\wec{z})$.
By specializing if necessary we may assume further
that $w_j, z_j\in\{x,y\}$ for all $j$. Our goal is to show that
$t^{\m s}(x_1,\ldots,x_n)$ 
is independent of its first variable.

\begin{clm}
  For any $s\in S$, $t^{\m s}(s,0,0,\ldots,0)=0$.
\end{clm}

\cproof
The identity $t(x,\wec{w})=t(y,\wec{z})$ may be written
symbolically as
\[t((x,y),(\wec{w},\wec{z}))\in D,\]
where
$(x,y)$ and each $(w_j,z_j)$ belong to the set
$\{(x,y), (y,x), (x,x), (y,y)\}$.

Choose $s\in S$ and represent it as
$s = (a,b)/\Delta$ for some pair $(a,b)\in\theta$.
Each of the pairs $(a,b), (b,a), (a,a), (b,b)$
belongs to $\theta$, so we may substitute $a$'s and $b$'s
for $x$'s and $y$'s to obtain that
\[
t^{\m a(\theta)}((a,b),(\wec{c},\wec{d}))\in D,
\]
where each $(c_j,d_j)$ is one of the elements of
$\{(a,b), (b,a), (a,a), (b,b)\}$. 
Factoring by $\Delta$ yields
\begin{equation}\label{weird}
t^{\m s}\big((a,b)/\Delta,(\wec{c},\wec{d})/\Delta\big) = 
t^{\m s}\big(s,\underline{(\wec{c},\wec{d})/\Delta}\big) = 0.
\end{equation}
Now we
apply Property~P to 
the polynomial $p(\wec{y})=t^{\m s}(s,\wec{y})$
to change
the underlined values 
in (\ref{weird}) 
to $0$. We obtain that
$t^{\m s}(s,\wec{0})=0$, as desired.
\cqed

\bigskip
Recall that $\m s\in\mathcal{V}$ is abelian. Therefore,
for arbitrary $s\in S$, 
we may
apply the term condition to
\[
t^{\m s}(s,\underline{\wec{0}}) = t^{\m s}(0,\underline{\wec{0}}) =0
\]
to obtain
\[
t^{\m s}(s,\underline{\wec{u}}) = t^{\m s}(0,\underline{\wec{u}})
                                        \phantom{{}=0}
\]
for any $\wec{u}$. This is what it means for $t^{\m s}(x_1,\ldots,x_n)$
to be independent of its first variable.

\bigskip

[Item (5)]
Let $R$ be the reflexive compatible relation on $\m s$
generated by $\{0\}\times S$. 
Hence $R$ consists of all pairs $(p(\wec{0}),p(\wec{s}))$
where $\wec{s}$ is a tuple of elements of $\m S$
and $p$ is a polynomial of $\m s$.
Property~P asserts exactly that
$(x,0)\in R$ implies $x=0$. The transitive closure $R^*$ of $R$
also has this property. 
Therefore the
symmetrization $\sigma:= R^*\cap (R^*)^{\cup}$
is a congruence on $\m s$ and $\leq := R^*/\sigma$ is a compatible
partial order on the quotient
$\m s/\sigma$.
This partial order
contains $(\{0\}\times S)/(\sigma\times\sigma)$, so
$0\leq s$ for every 
$s\in S/\sigma$.

Note that $\m s/\sigma$ satisfies all of the earlier properties.
(1):~
The quotient $\m s/\sigma$ is nontrivial, since
$\m s$ is nontrivial and $\{0\}$ is a singleton class
of $\sigma$. (2):~
$\{0\}/\sigma$ is a singleton subuniverse of the quotient.
(3):~ Property~P is easily derivable from a
lower bounded compatible order:
$0\leq p(\wec{0})\leq p(\wec{s})$ for any $\wec{s}$,
so $p(\wec{s})=0$ implies $p(\wec{0})=0$.
(4):~ The 
assumption
of part (4) of the
Lemma statement 
depends on $\mathcal{V}$ 
only, while the conclusion is
preserved when taking quotients.
\end{proof}

The properties that have been proved for $\m s = \m s_{\m a,\theta}$
and its quotient $\m s/\sigma$
prevent $\mathcal V$ from being affine. For example, no nontrivial
affine algebra can satisfy Property~P:
let $p(x) = x-s$ for some $s\in S\setminus\{0\}$.
Then $p(s)=0$ while $p(0)\neq 0$. In fact,
this polynomial has no fixed points at all.

Similarly, an affine algebra has no compatible reflexive relations
other than equivalence
relations. If the compatible partial order
in (5) was an equivalence relation, then it would be discrete.
For the discrete order to have a least element $0$, the
underlying set could have only one element, contrary to item (1).

Also, it is not hard to show that 
a variety that contains
an algebra $\m s/\sigma$ satisfying
the property described in item (4) cannot satisfy 
any nontrivial idempotent Maltsev condition, while
affine varieties satisfy strong idempotent Maltsev
conditions
(they in fact have a Maltsev-term).
These observations justify the following definition.
\begin{df}\label{affine_obstruction}
An algebra $\m s$ is called an \emph{affine obstruction}
if it contains an element $0$ such that 
conditions (1)--(4) of Lemma~\ref{nonaffine} hold for 
$\m s$ and the variety
$\mathcal V$ generated by $\m S$.
\end{df}

\begin{thm}\label{affine}
  The following are equivalent for an abelian variety
  $\mathcal V$.
  \begin{enumerate}
  \item $\mathcal V$ is not affine.
  \item $\mathcal V$ satisfies no nontrivial idempotent Maltsev condition.
  \item $\mathcal V$ contains an 
    algebra
    that has a 
    nontrivial
    strongly abelian 
    congruence.
  \item $\mathcal V$ contains an affine obstruction.
    \end{enumerate}
  \end{thm}

\begin{proof}
  {[$(1)\Rightarrow(2)$]} 
  (First proof.)
  We argue the contrapositive, so assume that $\mathcal V$ satisfies
  a nontrivial idempotent Maltsev condition.
  By Theorem 4.16~(2) of \cite{kearnes-kiss}, congruence
  lattices of algebras
  in $\mathcal V$ omit pentagons with certain specified
  abelian intervals. Since $\mathcal V$ is abelian,
  all intervals in congruence lattices of members are abelian.
  Hence there are no pentagons in congruence lattices of
  members of $\mathcal V$, which means that $\mathcal V$ is congruence
  modular. In this context it is known that
  abelian varieties are affine (see \cite{freese-mckenzie}).

[$(1)\Rightarrow(2)$] (Second proof.)
  Again we argue the contrapositive, so assume that $\mathcal V$ satisfies
  a nontrivial idempotent Maltsev condition.
  By Theorem 3.21 of \cite{kearnes-kiss}, $\mathcal V$ has a join term.
  The join term
  acts as a semilattice operation on blocks of any rectangular tolerance
  of an algebra in $\mathcal V$. Since 
  every algebra in $\mathcal{V}$ is abelian and
  there are no
  nontrivial abelian semilattices, it follows that
  rectangular tolerances in $\mathcal V$ are
  trivial. (This fact can also be deduced from Corollary~5.15 of
  \cite{kearnes-kiss}.) Now by Theorem~5.25 of \cite{kearnes-kiss},
  it follows that $\mathcal V$ satisfies an idempotent Maltsev condition
  that fails in the variety of semilattices. By Theorem~4.10
  of \cite{kearnes-szendrei}, $\mathcal V$ is affine.

[$(2)\Leftrightarrow(3)$]
This is part of Theorem~3.13 of \cite{kearnes-kiss}.

[$(3)\Rightarrow(4)$]
If $\mathcal V$ contains an algebra $\m a$ with a 
nontrivial strongly
abelian congruence~$\theta$, then it contains 
$\m s = \m s_{\m a,\theta} :=\m a(\theta)/\Delta$,
which is an affine obstruction by Lemma~\ref{nonaffine}.

[$(4)\Rightarrow(1)$]
Here it suffices to prove that an affine obstruction
for $\mathcal V$ prevents $\mathcal V$ from being
affine. This was explained right after the proof
of Lemma~\ref{nonaffine}.
\end{proof}

\section{Varieties whose finitely generated members are free}
In this section we investigate the class of varieties 
whose finitely generated members are free. 
This class
of varieties is closed under definitional equivalence.
The symbol
$\mathcal V$ will be used only to denote some nontrivial member
of this class.
We shall divide our analysis of this class
into two cases: the subclass of 
varieties with no $0$-ary function symbols
versus 
the subclass of varieties with at least one $0$-ary function symbol.

We shall prove that if the
finitely generated members of $\mathcal V$ are free,
then $\mathcal V$ must be definitionally
equivalent to the variety of sets, pointed sets, vector
spaces over a division ring, or affine spaces over a division
ring. It is obvious that each of these varieties
has the property that its finitely generated members are free.

\subsection{Varieties without constants}\label{subsection1}
First we will consider the case when $\mathcal V$ has no $0$-ary
function symbols.
We may write the $m$-generated free algebra in $\mathcal V$
as $\m f_{\mathcal V}(m)$, or as
$\m f_{\mathcal V}(X)$ for some $m$-element set $X$.

\begin{thm}\label{idempotent}
  Assume that 
  $\mathcal V$ is a nontrivial variety such that the
  finitely generated algebras in $\mathcal V$ are free.
  If $\mathcal V$ 
  has no $0$-ary function symbols, then $\mathcal V$
  is definitionally equivalent
  to the variety of sets or to a variety of affine
  spaces over a division ring.
  \end{thm}

\begin{proof}
  If 
  $\mathcal V$ has no $0$-ary function symbols, then
  $\m f_{\mathcal V}(\emptyset)$ is empty. $\m f_{\mathcal V}(1)$
  is the only candidate for the $1$-element
  algebra in $\mathcal V$. Hence $\mathcal V$ is idempotent.

  It follows from the standard proofs of Magari's Theorem
  (\cite{bu-sa}, Theorem 10.13) that
  every nontrivial variety has a {\it finitely generated} simple member.
  A free algebra $\m f_{\mathcal V}(X)$ over $X = \{x_1,x_2,\ldots\}$
  cannot be simple if $|X|>2$, since there are
    noninjective homomorphisms
    $\varepsilon_i\colon \m f_{\mathcal V}(X)\to \m f_{\mathcal V}(y,z)$ 
defined on generators by 
  \begin{equation}\label{kernels}
    x_j\mapsto
    \begin{cases}
  y & \textrm{if $j=i$}\\
  z & \textrm{else}.
  \end{cases}   
  \end{equation}
  If $\mathcal V$ is idempotent, then $\m f_{\mathcal V}(X)$
  cannot be simple for $|X|<2$, either. Thus, in our situation
  $\m f_{\mathcal V}(2)$ is the only candidate for a finitely generated
  simple member of $\mathcal V$.

  Let $\mathcal M$ be a minimal subvariety of $\mathcal V$.
  $\mathcal M$ also must contain a finitely generated simple algebra,
  and $\m f_{\mathcal V}(2)$ is the only one in $\mathcal V$
  up to isomorphism, so
  $\mathcal M$ must contain (and be generated by)
  $\m f_{\mathcal V}(2)$. 
  Every finitely generated algebra
   $\m a\in \mathcal M$
   is finitely
  generated in $\mathcal V$, hence is
   free in $\mathcal V$,
   hence satisfies the universal mapping property in~$\mathcal V$
  relative to some subset $X\subseteq A$, 
   hence
 satisfies the universal mapping property in~$\mathcal M$
  relative to the same subset, 
  hence is
  free over the same free generating set in~$\mathcal M$.
  This shows that $\mathcal M$ is also a variety whose finitely generated
  algebras are free. Also, $\m f_{\mathcal M}(2)=\m f_{\mathcal V}(2)$.

   According to Corollary 2.10 of \cite{kearnes},
  any minimal idempotent variety, like $\mathcal M$,
  is definitionally equivalent to the
  variety of sets, the variety of semilattices, a variety
  of affine modules over a simple ring, or is congruence distributive.
  
  The variety of semilattices does not have
  the property that its finitely generated members are free.

  No minimal, congruence distributive, idempotent
  variety 
  $\mathcal{M}$ has the property
  that its finitely generated members are free,
  as we now explain. If otherwise, then since
  $\m f_{\mathcal M}(x,y)\times \m f_{\mathcal M}(x,y)$
  is finitely generated (by $\{x,y\}\times \{x,y\}$), it must be
  isomorphic to $\m F_{\mathcal M}(m)$ for some $m$. Since
  $\m f_{\mathcal M}(x,y)\times \m f_{\mathcal M}(x,y)$ is not trivial or simple,
we have $m>2$.   The homomorphisms $\{\varepsilon_i\}_{i=1}^m$
    described in (\ref{kernels})
  (with subscript $\mathcal M$ in place of~$\mathcal V$)
  map $\m f_{\mathcal M}(m)$
  onto the simple algebra $\m f_{\mathcal M}(2)$,
  and $\varepsilon_i$ has kernel
  different from that of $\varepsilon_j$ when $i\neq j$.
  Thus $\m f_{\mathcal M}(m)$ has at least $m$ distinct coatoms 
  of the form $\ker(\varepsilon_i)$
  in its congruence lattice.
  From this it follows that 
  $\m f_{\mathcal M}(x,y)\times \m f_{\mathcal M}(x,y)\cong \m f_{\mathcal M}(m)$,
  $m>2$,
  has at least
  $3$ coatoms in its congruence lattice.
  But in a congruence distributive variety, the square of a simple
  algebra has exactly two coatoms in its congruence lattice.

  Now consider the case where $\mathcal M$ is a
  variety of affine (left) modules over some ring $R$.
  One realization of $\m f_{\mathcal M}(2)$
  has universe $R$, generators $0, 1\in R$, and term operations
  of the form
  \[
r_1x_1+\cdots + r_hx_h,\quad r_i\in R, \quad \sum r_i = 1.
\]
Each left ideal of $R$ induces a congruence on this algebra.
Since $\m f_{\mathcal M}(2)$ is simple, $R$ can have no nontrivial
proper left ideals, hence $R$ must be a division ring.

We have thus far argued that if $\mathcal V$ has the property that
its finitely generated members are free, and $\mathcal M$ is a minimal
subvariety of $\mathcal V$, then $\mathcal M$ is definitionally
equivalent to the variety
of sets or a variety of affine modules over a division ring.
  We now argue that $\mathcal V = \mathcal M$. If this is not
  the case, then there is a finitely generated
  algebra in $\mathcal V\setminus\mathcal M$,
  which we may assume is $\m a:=\m f_{\mathcal V}(m)$.
  By its very definition, $\m a$ has an $m$-element generating set
  that is minimal under inclusion as a generating set.
  Now let $\m b$ be the $m$-generated free algebra in $\mathcal M$.
  So
  $\m b$ also has an $m$-element minimal generating set.
  Since
  $\m b\in{\mathcal M}$, 
  we get that
  $\m b\in {\mathcal V}$, 
  but
  $\m b$ cannot be isomorphic to $\m a$,
  because $\m a\notin\mathcal{M}$. Hence
  $\m b \cong \m f_{\mathcal V}(n)$
  for some $n\neq m$. This implies that $\m b$
  has an $n$-element minimal generating set as well as an
  $m$-element minimal generating set.
  But $\mathcal M$ is definitionally
  equivalent to the variety of sets or 
  to a
  variety
  of affine spaces
  over a division ring,
  so it is not possible for $\m b$ to have minimal
  generating sets of different cardinalities.
  We conclude that $\mathcal V = \mathcal M$.
    \end{proof}

\subsection{Varieties with constants}\label{subsection2}

We still assume that 
$\mathcal V$ is a nontrivial variety whose finitely generated
members are free. In this subsection we also assume
that $\mathcal V$ has $0$-ary function symbols in its language.
In this situation, 
$\m f_{\mathcal V}(\emptyset)$ must
be the $1$-element algebra in $\mathcal V$,
so there is only one constant up to equivalence.
We will assume that there is exactly one constant in the language and
use $0$ to denote it. In any algebra $\m a\in\mathcal V$
the set $\{0\}$ is the unique 1-element subuniverse of $\m a$.
We will refer to $0\in\m a$ as the \emph{zero element} of $\m a$.

In the situation we are in now, 
when $\m f_{\mathcal V}(\emptyset)=\{0\}$,
it is
$\m f_{\mathcal V}(1)$ rather than
$\m f_{\mathcal V}(2)$ that is the only candidate
for the finitely generated simple algebra of $\mathcal V$.
To see this, note that when $m$ is greater than $1$, then 
$\m f_{\mathcal V}(x_1,\ldots,x_m)$ has at least
three distinct kernels of homomorphisms onto $\m f_{\mathcal V}(x)$,
namely the kernels of the homomorphisms defined on generators by
  \begin{enumerate}
  \item $x_1\mapsto 0$; $x_2, \ldots, x_m\mapsto x$, 
  \item $x_1\mapsto x$; $x_2, \ldots, x_m\mapsto 0$, and
  \item $x_1, x_2, \ldots, x_m\mapsto x$.
  \end{enumerate}
  To see that the kernels of these homomorphisms are distinct,
  it suffices to note that they restrict differently to
  the set $\{0,x_1,\ldots,x_m\}\subseteq F_{\mathcal V}(x_1,\ldots,x_m)$.
  Thus $\m f_{\mathcal V}(m)$ cannot be simple when $m>1$, nor
  can it be simple when $m=0$,
  hence $\m f_{\mathcal V}(1)$ is the 
  finitely generated simple member of $\mathcal V$. This argument
  also shows that, if $m>1$, then $\m f_{\mathcal V}(m)$ has at least
  $3$ coatoms in its congruence lattice. We record these observations as:

  \begin{lm}\label{3coatoms}
    If $\mathcal V$ is a nontrivial variety with
    at least one $0$-ary function symbol in its language, and
    all finitely generated members of $\mathcal V$ are free, then
    \begin{enumerate}
    \item 
      $\m f_{\mathcal V}(\emptyset)$ 
      has one element.
    \item $\m f_{\mathcal V}(1)$ is simple.
    \item $\m f_{\mathcal V}(m)$ has at least $3$ distinct coatoms in
      its congruence lattice for every finite $m>1$.      $\Box$
      \end{enumerate} 
  \end{lm}

  Later we will need to remember that, from part (3) of this lemma,
  any finitely generated, nontrivial, nonsimple member of $\mathcal V$
  has at least 3 distinct coatoms in its congruence lattice.

Suppose that $\m a\in \mathcal V$ and $a\in A\setminus \{0\}$.
Then there is a homomorphism $\m f_{\mathcal V}(x)\to \m a$
mapping
$x\mapsto a$,
which cannot be constant
(since $0\mapsto 0$).
By the
simplicity of $\m f_{\mathcal V}(x)$, this homomorphism
must be 
injective.
This shows
that $a$ is a free generator of the subalgebra $\langle a\rangle\leq \m a$.
We record this as:

    \begin{lm}\label{freely}
    If\/ $\mathcal V$ is a nontrivial variety with
    a $0$-ary function symbol
    $0$, 
    and
    all finitely generated members of $\mathcal V$ are free, then
    any nonzero element of any 
    algebra in
    $\mathcal V$
    {freely} generates a subalgebra isomorphic to $\m f_{\mathcal V}(x)$. $\Box$
  \end{lm}

\begin{lm}\label{abelian}
    If\/ $\mathcal V$ is a nontrivial variety with
    a $0$-ary function symbol, and
    all finitely generated members of $\mathcal V$ are free, then
    $\m f_{\mathcal V}(x)$ is abelian.
  \end{lm}

\begin{proof}
  In this proof we will abbreviate $\m f_{\mathcal V}(x)$ by $\m f$.
  
  Let $\m a$ be the subalgebra of
  $\m f\times \m f$ that is generated
  by $(0,x)$ and $(x,0)$. Let $\eta_1, \eta_2\in\Con(\m a)$
  be the restrictions to $\m a$ of the 
  coordinate projection kernels. Observe that the $\eta_1$-class
  of $0^{\m a} = (0,0)$ is the set $\{0\}\times F$, which is a
  subuniverse of $\m a$ that supports a subalgebra isomorphic to 
  $\m f$;
  hence this subalgebra is simple. 
  Similarly, the $\eta_2$-class of $0^{\m a}$,
  $F\times \{0\}$, is the universe of a simple subalgebra of $\m a$.

  $\m a$ is generated by $(0,x)$ and $(x,0)$, so
  every class of the congruence $\cg\big((0,0),(0,x)\big)$ of $\m a$ 
  contains an element of $F\times\{0\}$. As 
  $\cg\big((0,0),(0,x)\big)$ is contained in
  $\eta_1$, and each $\eta_1$-class contains exactly one
  element of $F\times \{0\}$, it follows that
  $\eta_1 = \cg\big((0,0),(0,x)\big)$.

  This shows that
  $\eta_1$ is principal,
  hence compact, so
  there
  is a congruence $\mu$ that is maximal among congruences
  strictly below $\eta_1$. $\Conb(\m a/\mu)$ contains
  a $3$-element maximal chain $0 = \mu/\mu\prec \eta_1/\mu\prec 1$.
  We apply Lemma~\ref{3coatoms}~(3) to $\m a/\mu$: the algebra $\m a/\mu$
  is nontrivial, nonsimple, and a quotient of the $2$-generated
  algebra $\m a$, so it is finitely generated. The lemma guarantees
  that $\Conb(\m a/\mu)$ has at least $3$ coatoms. The congruence
  $\eta_1/\mu$ is a coatom, but there must be at least two
  other
coatoms,
 say $\alpha, \beta\in\Conb(\m a/\mu)$.

  Since $\alpha, \beta$ and 
  $\eta_1/\mu$ 
  are 
pairwise
incomparable congruences,
  and 
  $\eta_1/\mu$ 
  is an atom in $\Conb(\m a/\mu)$,
  we have $\alpha\wedge (\eta_1/\mu) = 0 = \beta\wedge (\eta_1/\mu)$.
  We also have
  \[
(\alpha\vee\beta)\wedge (\eta_1/\mu) = 
1 \wedge (\eta_1/\mu) = \eta_1/\mu,
\]
so the interval $[0,\eta_1/\mu]$ is a meet semidistributivity failure
in $\Conb(\m a/\mu)$. It follows
from basic properties of the commutator
that $\eta_1/\mu$ is abelian.

Recall that $\{0\}\times F$ is a subuniverse of $\m a$
that is an $\eta_1$-class.
The congruence $\mu$ is strictly
smaller than $\eta_1 = \cg\big((0,0),(0,x)\big)$,
so it does not contain $\{0\}\times F$ entirely within a class.
Since the subuniverse supported by $\{0\}\times F$
is isomorphic to $\m f=\m f_{\mathcal V}(1)$, and therefore simple,
$\mu$ restricts
trivially to this set. This implies that
$(\{0\}\times F)/\mu$
is a class of $\eta_1/\mu$ 
that supports a subalgebra
of $\m a/\mu$
isomorphic to $\m f$. Since $\eta_1/\mu$ is abelian, it follows
that $\m f$ is abelian too.
  \end{proof}

\begin{lm}\label{injective}
    If\/ $\mathcal V$ is a nontrivial variety with
    a $0$-ary function symbol, and
    all finitely generated members of $\mathcal V$ are free, 
    then the nonconstant unary polynomial operations
    of $\m f_{\mathcal V}(x)$ are injective.
\end{lm}

\begin{proof}
  We first show that the nonconstant unary term operations
  act injectively on 
  $\m f=\m f_{\mathcal V}(x)$.
  Here we use a symbol, say $r$, for both an element of $F$
  and also for a unary term operation $r^{\m f}$ that represents
  the element $r$, i.e. $r = r^{\m f}(x)$.
  If $r,s\in F$, we will use the notation $rs$ for $r^{\m f}(s)$.
  Thus, our goal is to show that if $r\in F\setminus\{0\}$,
  then $rs=rt$ implies $s=t$ for all $s,t\in F$.
  
  Let $\eta_1, \eta_2, \Delta\in\Conb(\m f\times \m f)$ be
  the coordinate projection kernels and the congruence obtained
  from collapsing the diagonal. Suppose that $r, s, t\in F$,
  and that $rs=rt$ while $s\neq t$.

  By Lemma~\ref{freely} the
  element $(s,t)\in F\times F$ freely generates a subalgebra
  of $\m f\times \m f$ that is
  isomorphic to $\m F_{\mathcal V}(1)$, hence it is
  a simple subalgebra 
that
we denote
by
 $\m T$.
  Since $\m f$ is abelian by Lemma~\ref{abelian}, we
  have 
that $(s,t)$ and $(0,0)$ are not $\Delta$-related,
 so $\Delta|_T$ is trivial.
  But 
$rs=rt$ implies that
$(rs,rt) \equiv_\Delta (0,0)$,
 so $(rs,rt)=(0,0)$. This shows that if
  $s\neq t$ and $rs=rt$, then $rs=0=rt$.
  At least one of $s$ and $t$ is not $0$,
  and the situation between $s$ and $t$ has been symmetric up to this point,
  so assume that $s\neq 0$.

  As before, the element $(s, x)$ generates a
  simple subalgebra $\m x$ of $\m f\times \m f$, since 
$x\neq 0$.
The assumption $s\ne 0$ implies that
$(s,x)$ and $(0,0)$ are not $\eta_1$-related.
Therefore
$\eta_1|_X$ is trivial.
  But $(rs,rx) \equiv_{\eta_1} (0,0)$, 
so $(rs,rx)=(0,0)$. 
  Hence $r=rx=0$.
  This shows that our statement holds for unary term operations of $\m f$.

  Now we generalize our conclusion from unary
  term operations to unary polynomial operations of $\m f$.

Assume that $p(x)=t^{\m f}(x,\wec{u})$
for some term $t$ and some tuple $\wec{u}$.
If $p(a)=p(b)$, then
\[
t^{\m f}(a,\underline{\wec{u}}) = t^{\m f}(b,\underline{\wec{u}}).
\]
$\m f$ is abelian by Lemma~\ref{abelian}, therefore the last displayed
equality
is equivalent to 
\[
t^{\m f}(a,\underline{\wec{0}}) = t^{\m f}(b,\underline{\wec{0}}),
\]
by the term condition.
This shows that the unary polynomial $p(x) = t^{\m f}(x,\wec{u})$
has the same kernel as the ``twin'' unary term operation
$t^{\m f}(x,\wec{0})$. But such kernels have been shown to
be trivial or universal in the first part of this proof,
so they remain so here. I.e., any nonconstant unary
polynomial operation acts injectively on $\m f$.
  \end{proof}

Now we are prepared to prove the main result of this subsection.

\begin{thm}
  Assume that 
  $\mathcal V$ is a nontrivial variety such that the
  finitely generated algebras in $\mathcal V$ are free.
  If\/ $\mathcal V$ 
  has at least one $0$-ary function symbol, then $\mathcal V$
  is definitionally equivalent to either the variety of
  pointed sets or a variety of vector spaces over a division ring.  
\end{thm}

\begin{proof}
  Let $\mathcal M$ be a minimal subvariety of $\mathcal V$.
  By the same argument we used in Theorem~\ref{idempotent},
  $\mathcal M$ also has the property that its finitely generated
  algebras are free. We first prove the theorem for $\mathcal M$,
  then lift the result to $\mathcal V$, as we did in Theorem~\ref{idempotent}.

  All the lemmas proved for $\mathcal V$ in this subsection
  hold for $\mathcal M$. In particular, 
\begin{enumerate}
\item[(i)]
$\mathcal M$ has only one $0$-ary function symbol, up to equivalence,
which we denote by $0$;
\item[(ii)]
$\{0\}$ is the unique $1$-element subalgebra in every member of $\mathcal M$, 
and
\item[(iii)]
the unique finitely generated
  simple algebra in $\mathcal M$,
up to isomorphism, is $\m f_{\mathcal M}(1) = \m f_{\mathcal V}(1)$.
\end{enumerate}
  By the minimality of $\mathcal M$,
  ${\mathcal M}=\Ho\Su\Pd(\m f_{\mathcal M}(1))$, and the free algebras
  of $\mathcal M$ therefore lie in $\Su\Pd(\m f_{\mathcal M}(1))$. This
  latter class contains all the free algebras of $\mathcal M$,
  hence contains all of the finitely generated members of $\mathcal M$,
  hence generates $\mathcal M$ as a universal class:
\begin{equation}\label{Mvariety}
  {\mathcal M} =
  \Su\Pd_U(\Su\Pd(\m f_{\mathcal M}(1)))=\Su\Pd\Pd_U(\m f_{\mathcal M}(1)).
  \end{equation}
  By Lemma~\ref{abelian},
  $\m f_{\mathcal M}(1)$ is abelian, hence from (\ref{Mvariety})
  we deduce that $\mathcal M$ is an abelian variety.

  As a first case, assume that $\mathcal M$ is affine. 
 It follows from facts (i) and (ii) above and Lemma~4.3 of 
 \cite{szendrei-qlin} that 
 $\mathcal M$ is definitionally equivalent to a variety of
  left $R$-modules for some ring $R$. One realization of 
  $\m f_{\mathcal M}(1)$ has universe $R$, generator $1$, and term operations
  of the form
  \[
r_1x_1+\cdots + r_hx_h,\quad r_i\in R.
\]
Each left ideal of $R$ induces a congruence on this algebra.
Since $\m f_{\mathcal M}(1)$ is simple, $R$~can have no nontrivial
proper left ideals, hence $R$ must be a division ring.

For the remaining case we may assume, from Theorem~\ref{affine},
that $\mathcal M$ has
an affine obstruction
$\m s$
(see Definition~\ref{affine_obstruction}).
The element of $S$
referred to as
$0$ in Definition~\ref{affine_obstruction}
is a singleton subuniverse of $\m S$, therefore fact (ii)
ensures that it
must be the element named by our constant symbol $0$.
It is easy to see that any nontrivial
subalgebra of an affine obstruction
$\m S$
which contains $0$ is again an affine obstruction
(i.e., inherits properties (1)--(4)
of Lemma~\ref{nonaffine}). Since
we know from Lemma~\ref{freely} that
every nontrivial $1$-generated 
subalgebra of $\m S$
is isomorphic
to $\m f_{\mathcal M}(1)$, 
we conclude that $\m f_{\mathcal M}(1)$ has Property~P.

\begin{clm}\label{claim}
$\m f_{\mathcal M}(1)$ has size $2$.
  \end{clm}

\cproof
Assume otherwise that there are distinct nonzero
elements 
$a, b$ in  $F_{\mathcal M}(1)$.
 The congruence
$\cg(a,b)$ is 
nontrivial,
hence by the simplicity of
$\m f_{\mathcal M}(1)$
there is a unary polynomial $p(x)$
of $\m f_{\mathcal M}(1)$ such that $p(a)=0\neq p(b)$, or the same
with $a$ and $b$ interchanged. But $p(a)=0$ implies
$p(0)=0$, by Property~P, showing that $(a,0)$
is a nontrivial pair in $\ker(p)$. On the other hand
$(a,b)$ is a pair not in $\ker(p)$. This contradicts
Lemma~\ref{injective}, which establishes that unary polynomials
of $\m f_{\mathcal M}(1)$ are constant or injective.
  \cqed

  \bigskip
  
  Claim~\ref{claim}, together with earlier information, yields that
  $\m f_{\mathcal M}(1)$ is a 2-element, nonaffine, abelian
  algebra with a singleton subalgebra named by a constant.
  There is one such algebra up to definitional 
  equivalence,
  namely the 2-element pointed set. (The simplest way to affirm this
  is to refer to Post's classification of 2-element algebras,
  but one doesn't need a result of such depth to make this conclusion.)

  Since $\mathcal M$ is generated by $\m f_{\mathcal M}(1)$, which
  is equivalent to a pointed set, it follows that ${\mathcal M}$
  is definitionally equivalent to the variety of pointed sets
  in the case we are considering.

  We have shown that $\mathcal M$ is definitionally equivalent
  to a variety of vector spaces over a division ring or the variety
  of pointed sets. 
  We now argue that ${\mathcal V}={\mathcal M}$ using the same type
  of argument used in Theorem~\ref{idempotent}.

  If ${\mathcal V}\neq {\mathcal M}$,
  there is a finitely generated
  algebra in $\mathcal V\setminus\mathcal M$,
  which we may assume is $\m a:=\m f_{\mathcal V}(m)$.
  Then
 $\m a$ has an $m$-element generating set
  that is minimal under inclusion as a generating set.
  Let $\m b$ be the $m$-generated free algebra in
  $\mathcal M$.
The algebra 
  $\m b$ also has an $m$-element minimal generating set.
But
  $\m b\in{\mathcal M}$, so $\m b\in {\mathcal V}$, and
  $\m b$ cannot be isomorphic to $\m a$,
  so $\m b \cong \m f_{\mathcal V}(n)$
  for some $n\neq m$. This implies that $\m b$
  has an $n$-element minimal generating set as well as an
  $m$-element minimal generating set.
  But there does not exist a vector space nor a pointed set
  that has minimal generating sets of different cardinalities.
  We conclude that $\mathcal V = \mathcal M$.  
\end{proof}

\section{Discussion}

  Throughout this paper our arguments depended on some
  strong but odd assumptions, namely that a
  $1$-element $\mathcal V$-algebra
  is free and that a finitely generated simple
  $\mathcal V$-algebra is free. One might wonder whether
  anything can be proved for varieties where only
  the ``large'' finitely generated algebras are assumed to be free.
  Specifically, one might ask what can be said about the
  varieties $\mathcal V$ satisfying the following property:
  There exists a natural number $k$ such that every finitely generated
  algebra in $\mathcal V$ is either free or can be generated by $\leq k$
  elements.

  Unfortunately there is a seemingly-unclassifiable
  collection of varieties for which
  $\m f_{\mathcal V}(j)\cong\m f_{\mathcal V}(k)$ for some $j<k$.
  For any given $j<k$ the varieties with this property
  represent a filter in the lattice of interpretability types.
  In such varieties every finitely generated algebra can be generated
  by $\leq k$ elements, so the conditions of the question are satisfied.
  This suggests that there is no nice classification of
  the varieties $\mathcal V$ satisfying the property above.

  However, if we restrict our attention to locally finite varieties,
  then we can prove the following.

  \begin{thm}
    Let $\mathcal V$ be a nontrivial locally finite variety.
    If there exists a natural number $k$ such that every finitely generated
  algebra $\mathcal V$ is either free or can be generated by $\leq k$
  elements, then every nonsingleton algebra in $\mathcal V$ is free.
  In fact, $\mathcal V$ is definitionally equivalent to
\begin{enumerate}
\item the variety of sets,
\item the variety of pointed sets,
\item a variety of vector spaces over a finite field, or
\item a variety of affine spaces over a finite field.
\end{enumerate}
    \end{thm}

  \emph{Caveat:} While in the earlier part of the paper
  our ``pointed sets'' and ``vector spaces'' each had a
  (unique) $0$-ary term operation, in this theorem we
  allow the constants of the algebras in cases (2) and (3)
  to be constant $0$-ary term operations \emph{or} constant
  $1$-ary term operations.
  If these constants are $1$-ary term operations and
  there are no constant $0$-ary term operations, then no
  1-element algebra of the variety
  is free, but all the other algebras are free.
  
  \begin{proof}
    First observe that any variety satisfying the hypotheses
    of the theorem
    must be a minimal variety. For if ${\mathcal M}$
    is a minimal subvariety of ${\mathcal V}$, then the sequence
    $(\m f_{\mathcal M}(p))_{p\in\omega}$ consists of algebras
    in $\mathcal V$ whose sizes increase with $p$, and which require
    more generators as $p$ increases. It follows
    from the hypotheses of the theorem that some tail
    end of this sequence is cofinal in the sequence
    $(\m f_{\mathcal V}(q))_{q\in\omega}$. Hence the algebras
    in the first sequence 
    generate the same variety as the algebras in the second sequence,
    i.e. $\mathcal M = \mathcal V$.
        
    By local finiteness, the hypotheses on $\mathcal V$ ensure that
    there are at most finitely many (say $C$) isomorphism types
    of finitely generated non-free algebras in $\mathcal V$.
    Local finiteness ensures that $\m f_{\mathcal V}(n)$ cannot be
    $m$-generated if $m<n$. Hence if $n\geq C$, it follows that
    there are $n+1$ free algebras that can be generated by $\leq n$
    elements ($\m f_{\mathcal V}(0),\ldots,\m f_{\mathcal V}(n)$)
    and $C$ non-free algebras that can be generated by
    $\leq n$ elements, hence a total of $n+C$ algebras in $\mathcal V$
    that can be generated by $\leq n$ elements. This says
    precisely that the \emph{$G$-spectrum} of $\mathcal V$
    satisfies $G_{\mathcal V}(n) = n+C$ whenever $n\geq C$.
    (The $G$-spectrum of a locally finite variety
    $\mathcal V$ is the function whose value at $n$ is
    the number of isomorphism types of algebras in $\mathcal V$
    that can be generated by $\leq n$ elements.)

    It is known that a locally finite variety $\mathcal V$ whose
    $G$-spectrum $G_{\mathcal V}(n)$ is bounded above
    by a polynomial function of $n$ must be abelian
    (\cite{idziak-mckenzie-valeriote}, Theorem 8.15).
    So at this point we know that our variety $\mathcal V$
    is a locally finite, minimal, abelian variety.
    These have been classified in \cite{kkv1,szendrei,szendrei2}.
    Such varieties are definitionally equivalent to either
    a matrix power of the variety of sets,
    a matrix power of a variety of pointed sets (note the caveat
    between the theorem statement and the start of the proof),
    or to an affine variety over a finite simple ring where each member
    has a     singleton subuniverse.
    We will complete the proof of the theorem by examining the clones
    of such algebras.

Let $\m s$ be a strictly 
simple generator of our locally finite, minimal, abelian variety
$\mathcal V$.
It follows from the results in \cite{kkv1,szendrei,szendrei2}
that $\m s$ is isomorphic to an algebra that is term equivalent to 
(i.e., has the same underlying set and the same
non-nullary term operations as) one of the 
following algebras:
\begin{enumerate}
\item[(i)]
$\m a=({\bf 2};\emptyset)^{[d]}$ ($d\ge1$),
the $d$-th matrix power of the
$2$-element set ${\bf 2}=\{0,1\}$;
\item[(ii)]
$\m a=({\bf 2};0)^{[d]}$ ($d\ge1$),
(the $d$-th matrix power of the 
$2$-element pointed set $({\bf 2};0)$;
\item[(iii)]
an affine reduct $\m a$ of a finite simple module $\m m$ such that
$\m a$ has the same ring as $\m m$.  
\end{enumerate}
In each one of these cases, the fact that $\m s$ generates
$\mathcal V$ implies that
\[
|\m f_{\mathcal V}(n)|=|{\rm Clo}_n(\m a)|\quad \text{for every $n\ge1$},
\]
where ${\rm Clo}_n(\m a)$ denotes the set of $n$-ary term operations of $\m a$
(the $n$-ary sort of the clone of $\m a$).
Thus, if $\mathcal V$ satisfies the assumptions of the theorem, then
the (increasing) sequence of all sizes of finite algebras in $\mathcal V$
must have the same tail end as the sequence 
$(|{\rm Clo}_n(\m a)|)_{0<n<\omega}$.
To finish the proof of the theorem, we have to deduce from this condition
that
\begin{itemize}
\item
$d=1$ in cases (i) and (ii), and
\item
$\m a$ is a $1$-dimensional vector space or affine space over a finite field
in case (iii).
\end{itemize}

{\bf Cases (i)--(ii).}
Every operation $f\in{\rm Clo}_n(\m a)$ has the form
\begin{multline*}
f\colon({\bf 2}^d)^n\to{\bf 2}^d,\\
\bigl((x_{0,0},\dots,x_{0,d-1}),\dots,(x_{n-1,0},\dots,x_{n-1,d-1})\bigr)
\mapsto \bigl(f_0(x_{i_0,j_0}),\dots,f_{d-1}(x_{i_{d-1},j_{d-1}})\bigr) 
\end{multline*}
where, for each $\ell$, either $f_{\ell}=\id$ and $(i_\ell,j_{\ell})$
is a pair of integers with $0\le i_\ell<n$, $0\le j_\ell<d$, or
we are in case (ii) and $f_\ell$ is the (unary) constant operation with
value $0$ and the pair $(i_\ell,j_{\ell})$ is irrelevant.
It is easy to check that different choices yield different operations.
Hence $|{\rm Clo}_n(\m a)|=(nd)^d$ in case (i) and
$|{\rm Clo}_n(\m a)|=(nd+1)^d$ in case (ii).

For every finite set $B$ with $0\in B$, 
the algebra 
$(B;\emptyset)^{[d]}$ belongs to the variety generated by 
$({\bf 2};\emptyset)^{[d]}$, and
the algebra 
$(B;0)^{[d]}$ belongs to the variety generated by 
$({\bf 2};0)^{[d]}$.
Hence, ${\mathcal V}$ contains algebras of sizes $m^d$ for every
$m\ge 1$. Since our assumptions force that
the (increasing) sequence of all sizes of finite algebras in $\mathcal V$
has the same tail end as the sequence $(|{\rm Clo}_n(\m a)|)_{0<n<\omega}$,
we get that
a tail end of the sequence $(m^d)_{0<m<\omega}$ must be a subsequence
of a tail end of the sequence $\bigl((nd)^d\bigr)_{0<n<\omega}$ or 
$\bigl((nd+1)^d\bigr)_{0<n<\omega}$, according to whether we are in case 
(i) or (ii).
It is easy to see that in both cases this will hold only if $d=1$.

{\bf Case (iii).}
Let $\m a$ be an affine reduct of a finite, simple $R$-module 
such that the ring of $\m a$ is also $R$.
Since we are only interested in the term operations of $\m a$, we
may assume without loss of generality that $R$ and $\m m$ are unital
and $\m m$ is a faithful $R$-module.
Since $\m m$ is finite and simple, it follows that there
exist a finite field $K$ and a positive integer $d$ such that $R$ is
the ring of $d\times d$ matrices with entries in $K$, and
$\m m$ is a $d$-dimensional $K$-vector space with the usual action
of $R$ as an $R$-module.

Since $\m a$ is an affine reduct of $\m m$ with the same ring $R$ as $\m m$,
Lemma~4.3 of \cite{szendrei} implies that 
there exists a left ideal $L$ of $R$ such that
\begin{equation}\label{affine_clone}
{\rm Clo}_n(\m a)=\left\{\sum_{i=0}^{n-1}r_ix_i: r_0,\dots,r_{n-1}\in R\ \ 
\text{and}\ \ 1-\sum_{i=0}^{n-1}r_i\in L\right\}
\quad
\text{for all $n\ge1$.}
\end{equation}
Thus, $|{\rm Clo}_n(\m a)|=|R|^{n-1}|L|=|M|^{d(n-1)}|L|=|A|^{d(n-1)}|L|$ 
for all $n\ge1$.
The variety ${\mathcal V}$ contains finite algebras of sizes
$|A|^m=|S|^m$ for every $m\ge1$. Now, if $d>1$, then
no tail end of the sequence
$(|A|^m)_{0<m<\omega}$ is a subsequence of any tail end of the sequence
$(|{\rm Clo}_n(\m a)|)_{0<n<\omega}=(|A|^{d(n-1)}|L|)_{0<n<\omega}$. Therefore
we conclude the same way as before that $d=1$.
This implies that $R=K$ and 
$\m m$ is a $1$-dimensional $K$-vector space.
Hence, either $L=K$ or $L=\{0\}$, which implies by \eqref{affine_clone}
that $\m a$ is term equivalent to either the vector space $\m m$, or
the corresponding affine space (i.e., the full idempotent reduct
of $\m m$).  
  \end{proof}

  Now we turn to the opposite type of question:
  what can one say about the varieties for which
  there is a natural number $n$ such that every
  $(\leq n)$-generated algebra is free? If $n$ is large enough,
  must all algebras in the variety be free? We show that the answer
  to this is negative for any natural number $n$.

\begin{thm}\label{smallfree}
  For any natural number~$n$ there exists a variety with the
  property that every $(\leq n)$-generated algebra is free,
  but some $(n+1)$-generated algebra in the variety is not
  free.
\end{thm}

\begin{proof}
An $(m+1)$-ary (first variable) \emph{semiprojection} 
  on a set $A$ is an 
$(m+1)$-ary operation 
$s(x_0,x_1,\ldots,x_m)$
  on $A$ such that for any $\wec{a}\in A^{m+1}$ we have
  \[
s(a_0,a_1,\ldots,a_m) = a_0
\]
whenever $a_i=a_j$ for some $i\neq j$. 
This property can be
expressed by identities, so
starting with any variety $\mathcal V$ we can add an
$(m+1)$-ary function symbol $s$ to the language and define
${\mathcal V}_s$ to be the variety of all
$\mathcal V$-algebras expanded by an $(m+1)$-ary 
(first variable)
semiprojection.

The added semiprojection operation acts like first projection
on any algebra in~${\mathcal V}_s$ that has cardinality
at most $m$. Hence any algebra of size at most $m$
in ${\mathcal V}_s$ is definitionally equivalent to an algebra
in $\mathcal V$.

If $\mathcal V$ is the variety of sets, then this construction with $m=n$
yields a variety~${\mathcal V}_s$ in which
every algebra that is generated by at most $n$ elements
will be definitionally equivalent to a set, hence will
be free.
 Now let $\m B$ be the $(n+1)$-element algebra in~${\mathcal
   V}_s$ where $s$ interprets as a first projection, so $\m
 B$ is definitionally equivalent to a set. This algebra is
 not free, because there exist $(n+1)$-generated algebras
 in~$\mathcal V_s$ that are not homomorphic images of $\m B$.
For example, any $(n+1)$-element
algebra $\m a$ in~${\mathcal V}_s$ where
$s$ is a (first variable) semiprojection 
other than a projection has this property.

Similarly, if $\mathcal V$ is the variety of vector spaces over
the $2$-element field, and we let $m=2^n$, then
the $(2^n+1)$-ary semiprojection $s$
acts like first projection
on any algebra in ${\mathcal V}_s$ 
generated 
by at most $n$
elements. Again, all algebras in ${\mathcal V}_s$
that are generated by at most $n$ elements will be free,
but there will be $(n+1)$-generated algebras in ${\mathcal V}_s$
that are not free.
\end{proof}

\bibliographystyle{plain}

\end{document}